\documentclass[11pt]{amsart}

\usepackage{microtype}
\usepackage{amsmath,amssymb,amsthm}

\newcommand{\Z}{\mathbb{Z}}

\newcommand{\abs}[1]{\lvert#1\rvert}
\newcommand{\one}{\mathbf{1}}

\numberwithin{equation}{section}

\theoremstyle{plain}
\newtheorem{theorem}{Theorem}
\newtheorem{corollary}[theorem]{Corollary}
\newtheorem{proposition}[theorem]{Proposition}
\newtheorem{lemma}[theorem]{Lemma}
\theoremstyle{remark}

\title{A Fourier-Free Density-Increment Proof of Roth's Theorem}
\author{Mark Lewko}
\date{}

\begin{document}

\begin{abstract}
We give an elementary, Fourier-free proof of Roth's theorem. The proof
follows Roth's original density-increment strategy, but replaces the usual
Fourier-analytic step with a direct combinatorial argument involving averages
over sub-progressions.
\end{abstract}

\maketitle

%% -------------------------------------------------------------------
\section{Introduction}
%% -------------------------------------------------------------------

In 1953, Roth proved that every subset of the integers of positive upper density
contains a nontrivial three-term arithmetic progression. This can be stated in
the following finite form.

\begin{theorem}[Roth \cite{Roth53}]\label{thm:roth}
For every $\alpha>0$, if $N$ is sufficiently large, then every set
$A\subset\{0,1,\ldots,N-1\}$ with $\abs{A}\geq\alpha N$ contains a
three-term arithmetic progression $x,\,x+r,\,x+2r$ with $r\neq 0$.
\end{theorem}

Roth's theorem is a foundational result in density Ramsey theory; its extension
to longer arithmetic progressions is a celebrated theorem of Szemer\'edi ~\cite{Szemeredi75}.
Proofs of Roth's theorem have been given using quite different methods: Roth's original
proof is Fourier analytic~\cite{Roth53}; Furstenberg proved Szemer\'edi's
theorem in full using ergodic-theoretic methods~\cite{Furstenberg77}; and graph-theoretic
proofs of Roth's theorem follow from the triangle removal lemma of Ruzsa and
Szemer\'edi~\cite{RS78}. More combinatorial proofs have also been given using Szemer\'edi's cube
lemma~\cite{GRS} as well as using ideas related to the more general Hales--Jewett theorem~\cite{DHJ}. There is
also a rich history of quantitative improvements, which we do not attempt to
summarize here; see~\cite{Green,P} for recent surveys.

Roth's original proof uses a density-increment argument. Starting with a set
$A\subset[N]$ with no nontrivial three-term progressions, one proves that $A$
has increased density on a long subprogression. Passing to that progression
and rescaling, one repeats the argument. If the original interval is sufficiently
long, the density would eventually have to exceed~$1$, providing a contradiction.
The purpose of this note is to give a proof with the same density-increment
structure, but without using Fourier analysis. Instead of detecting a large
Fourier coefficient, we estimate certain averages over subprogressions directly.

The combinatorial approaches to Roth's theorem have typically produced very weak
quantitative bounds. The argument here stays close to Roth's density-increment
framework and gives a bound of the form
$$
        r_3(N)\lesssim \frac{N}{(\log\log N)^c}
$$
with $c=1/11$. This can be improved somewhat at the expense of complicating the exposition, but the argument in its current form does not seem capable of recovering the exponent $c=1$ obtained by Roth's original argument.

%% -------------------------------------------------------------------
\section{Notation}
%% -------------------------------------------------------------------

We write $A\lesssim B$ to mean $A\leq CB$ for a positive absolute constant
$C$, and similarly $A\gtrsim B$. We write $A\sim B$ when both $A\lesssim B$
and $B\lesssim A$ hold. All implicit constants are absolute. For a prime
$m$, write $\Z_m=\Z/m\Z$ and $\Z_m^\times=\Z_m\setminus\{0\}$.

%% -------------------------------------------------------------------
\section{Preliminaries}
%% -------------------------------------------------------------------

Write $[N]=\{0,\ldots,N-1\}$ and choose a prime $m$ with $4N<m<8N$, which we
may do by Bertrand's postulate. We embed $[N]$ in $\Z_m$ in the obvious way.
The point of choosing $m>2N$ is that a three-term progression in $\Z_m$ whose
terms all lie in $[N]$ is a genuine integer progression.

Let $\alpha=\abs{A}/N$ be the density of $A$ in $[N]$, and define
$f=\one_A-\alpha\one_{[N]}$ on $\Z_m$. Then $f$ has mean zero on $\Z_m$,
satisfies $-\alpha\leq f\leq 1$, and is supported on $[N]$. For any function
$g\colon\Z_m\to\mathbb R$, set
$$
\mathcal E(g)=\sum_{t\in\Z_m}\left(\sum_{x\in\Z_m}g(x)g(x+t)\right)^2.
$$
For the balanced function $f$, write $R(t)=\sum_{x\in\Z_m}f(x)f(x+t)$, so
$\mathcal E(f)=\sum_{t\in\Z_m}R(t)^2$. We will use the following
non-negativity property.

\begin{lemma}\label{lem:pos}
For any $h,k\in\Z_m^\times$, one has
\begin{equation}\label{eq:positivity}
\sum_{d\in\Z_m}R(hd)R(kd)\geq 0.
\end{equation}
\end{lemma}

\begin{proof}
Let $\lambda=kh^{-1}$. Using $t=hd$, we have
\[
\sum_{d\in\Z_m}R(hd)R(kd)=\sum_{t\in\Z_m}R(t)R(\lambda t)
=\sum_{a,b,t\in\Z_m} f(a)f(a+t)f(b)f(b+\lambda t).
\]
To make the non-negativity explicit, use the bijective change of variables
$(a,b,t)\mapsto(a,c,u)$ given by $c=a+t$ and $u=b-\lambda a$. Then
$b=u+\lambda a$ and $b+\lambda t=u+\lambda c$, so the last sum becomes
$$
\sum_{u,a,c\in\Z_m} f(a)f(c)f(u+\lambda a)f(u+\lambda c)
=\sum_{u\in\Z_m}\left(\sum_{a\in\Z_m}f(a)f(u+\lambda a)\right)^2\geq 0.
$$
\end{proof}

%% -------------------------------------------------------------------
\section{No progressions imply large energy}
%% -------------------------------------------------------------------

For functions $f_1,f_2,f_3\colon\Z_m\to\mathbb R$, define
$$
\Lambda(f_1,f_2,f_3):=\sum_{x,r\in\Z_m} f_1(x)f_2(x+r)f_3(x+2r).
$$
If $A\subset [N]\subset\Z_m$, then $\Lambda(\one_A,\one_A,\one_A)$ is the
number of three-term arithmetic progressions in $A$, including the trivial
progressions with $r=0$. Hence, if $A$ contains no nontrivial three-term
progressions, then $\Lambda(\one_A,\one_A,\one_A)=\abs{A}$. On the other hand,
$\Lambda(\one_{[N]},\one_{[N]},\one_{[N]})\gtrsim N^2$, since $[N]$ contains
$\gtrsim N^2$ three-term progressions.

\begin{lemma}\label{lem:lambda-energy}
Let $g_1,g_2,g_3\colon\Z_m\to\mathbb R$ be bounded in absolute value by an
absolute constant. Then
\[
        \abs{\Lambda(g_1,g_2,g_3)}^4
        \lesssim m^5\min_{1\leq i\leq3}\mathcal E(g_i).
\]
\end{lemma}

\begin{proof}
It suffices to prove the following estimate. Let $F,h_1,h_2\colon\Z_m\to\mathbb R$,
with $h_1,h_2$ bounded, and let $a,b\in\Z_m^\times$ with $a\neq b$. Then
\[
        \left|\sum_{y,r\in\Z_m}F(y)h_1(y+ar)h_2(y+br)\right|^4
        \lesssim m^5\mathcal E(F).
\]
The three possible choices of the distinguished slot in $\Lambda$ are obtained,
after a change of variables, from this estimate by taking $(a,b)=(1,2)$,
$(-1,1)$, and $(-2,-1)$.

Let $T$ denote the sum inside the fourth-power. Put $\lambda=ba^{-1}$ and
change variables from $r$ to $u=y+ar$. Then
\[
T=\sum_{y,u\in\Z_m}F(y)h_1(u)h_2((1-\lambda)y+\lambda u).
\]
By Cauchy--Schwarz in $u$,
\[
\abs{T}^2\lesssim m\sum_{u\in\Z_m}
\left|\sum_{y\in\Z_m}F(y)h_2((1-\lambda)y+\lambda u)\right|^2.
\]
Expanding the square, then putting $t=y'-y$ and $z=(1-\lambda)y+\lambda u$, gives
\[
\begin{aligned}
\abs{T}^2
&\lesssim m\sum_{t\in\Z_m}
\left(\sum_{y\in\Z_m}F(y)F(y+t)\right) \\
&\qquad\qquad\times
\left(\sum_{z\in\Z_m}h_2(z)h_2(z+(1-\lambda)t)\right).
\end{aligned}
\]
A final Cauchy--Schwarz in $t$ gives
\[
\begin{aligned}
\abs{T}^2
&\lesssim m\mathcal E(F)^{1/2}
\left(\sum_{t\in\Z_m}
\left(\sum_{z\in\Z_m}h_2(z)h_2(z+(1-\lambda)t)\right)^2
\right)^{1/2} \\
&\lesssim m^{5/2}\mathcal E(F)^{1/2},
\end{aligned}
\]
using that $|h_2|\leq 1$. Squaring proves the estimate. Applying this with each
$g_i$ as the distinguished function gives the minimum on the right.
\end{proof}

\begin{lemma}[Energy controls the trilinear operator]\label{lem:trilinear-discrepancy}
Let $u,v\colon\Z_m\to\mathbb R$ be bounded in absolute value by an absolute
constant, and put $f=u-v$. Then
\[
        \abs{\Lambda(u,u,u)-\Lambda(v,v,v)}^4\lesssim m^5\mathcal E(f).
\]
\end{lemma}

\begin{proof}
We telescope one slot at a time:
\[
\begin{aligned}
\Lambda(u,u,u)-\Lambda(v,v,v)
&=\bigl(\Lambda(u,u,u)-\Lambda(v,u,u)\bigr)
 +\bigl(\Lambda(v,u,u)-\Lambda(v,v,u)\bigr)  \\
&\qquad
 +\bigl(\Lambda(v,v,u)-\Lambda(v,v,v)\bigr)  \\
&=\Lambda(f,u,u)+\Lambda(v,f,u)+\Lambda(v,v,f).
\end{aligned}
\]
Each term is bounded by $m^{5/4}\mathcal E(f)^{1/4}$ by
Lemma~\ref{lem:lambda-energy}, since in each term the other two factors are
bounded. The claim follows after raising to the fourth power.
\end{proof}

\begin{lemma}[Lack of progressions implies large energy]\label{lem:energy}
If $A\subset[N]$ is 3AP-free and $N\gtrsim\alpha^{-3}$, then
$$
\mathcal E(f)\gtrsim\alpha^{12}m^3.
$$
\end{lemma}

\begin{proof}
Put $u=\one_A$ and $v=\alpha\one_{[N]}$. Then $f=u-v$ and
$\abs{u},\abs{v}\leq1$. Since $A$ has no nontrivial three-term progressions,
$\Lambda(u,u,u)=\abs{A}=\alpha N$, while
$\Lambda(v,v,v)=\alpha^3\Lambda(\one_{[N]},\one_{[N]},\one_{[N]})\gtrsim\alpha^3N^2$.
If $N\gtrsim\alpha^{-3}$, then
$\abs{\Lambda(u,u,u)-\Lambda(v,v,v)}\gtrsim\alpha^3N^2$.
Lemma~\ref{lem:trilinear-discrepancy} therefore gives
$\alpha^{12}N^8\lesssim m^5\mathcal E(f)$. Since $m\sim N$, this gives
$\mathcal E(f)\gtrsim\alpha^{12}m^3$.
\end{proof}

%% -------------------------------------------------------------------
\section{Large energy implies a modular density increment}
%% -------------------------------------------------------------------

Our next goal is to find a long modular progression on which $f$ has positive
average. For an integer $1\leq \ell<m/2$ and a step $d\in\Z_m$, set
$$
S_d(x):=\sum_{i=0}^{\ell-1}f(x+id),\qquad V_d:=\sum_{x\in\Z_m}\abs{S_d(x)}^2.
$$
Thus $S_d(x)$ is the sum of $f$ on the length-$\ell$ modular progression
starting at $x$ with step $d$, and $V_d$ is the second moment of these sums
as $x$ varies.

\begin{lemma}[Moments of the progression sums]\label{lem:v-identity}
For every $d\in\Z_m$,
\begin{equation}\label{eq:v-identity}
V_d=\sum_{\abs{h}<\ell}(\ell-\abs{h})R(hd).
\end{equation}
Also
\begin{equation}\label{eq:v-first-moment}
\sum_{d\in\Z_m}V_d\leq \alpha\ell m^2.
\end{equation}
Finally, if $\mathcal E(f)\geq\beta m^3$ and $\ell\lesssim \beta m/\alpha^2$, then
\begin{equation}\label{eq:v-second-moment-nonzero}
\sum_{d\in\Z_m^\times}V_d^2\gtrsim \beta\ell^3m^3.
\end{equation}
\end{lemma}

\begin{proof}
We first prove \eqref{eq:v-identity}. Expanding the square gives
\[
V_d = \sum_{x\in\mathbb Z_m}\left(\sum_{i=0}^{\ell-1}f(x+id)\right)^2
 = \sum_{0\leq i,j<\ell}\sum_{x\in\mathbb Z_m}f(x+id)f(x+jd)  \]

 \[= \sum_{0\leq i,j<\ell}\sum_{y\in\mathbb Z_m}f(y)f(y+(j-i)d)
 = \sum_{0\leq i,j<\ell} R((j-i)d) = \sum_{|h|<\ell}(\ell-|h|)R(hd).
\]
Here we used the change of variables \(y=x+id\) and then used that there are \(\ell-|h|\) pairs \((i,j)\)  satisfy \(0\leq i<\ell\) and
\(0\leq i+h<\ell\) such that \(h=j-i\). This proves \eqref{eq:v-identity}.

Now we sum \eqref{eq:v-identity} over $d\in\Z_m$. The term $h=0$ contributes
$\ell mR(0)$. If $h\neq0$, then $d\mapsto hd$ is a bijection on $\Z_m$, so
$$
\sum_{d\in\Z_m}R(hd)=\sum_{t\in\Z_m}R(t)=\left(\sum_{x\in\Z_m}f(x)\right)^2=0.
$$
Therefore $\sum_{d\in\Z_m}V_d=\ell mR(0)$. Since
$R(0)=\sum_{x\in\Z_m}f(x)^2\leq \alpha m$, this proves
\eqref{eq:v-first-moment}.

It remains to prove \eqref{eq:v-second-moment-nonzero}. Squaring
\eqref{eq:v-identity} and summing over $d\in\Z_m$ gives
$$
\sum_{d\in\Z_m}V_d^2=
\sum_{\abs{h},\abs{k}<\ell}(\ell-\abs{h})(\ell-\abs{k})
\sum_{d\in\Z_m}R(hd)R(kd).
$$
By Lemma~\ref{lem:pos}, every term with $h,k\neq0$ is nonnegative. Terms with
exactly one of $h,k$ equal to $0$ vanish by the preceding discussion. Thus we
may obtain a lower bound by considering just the diagonal terms $h=k\neq0$. For $h \neq 0$, multiplication by
$h$ is a bijection on $\Z_m$, so $\sum_{d\in\Z_m}R(hd)^2=\mathcal E(f)$. Hence
$$
\sum_{d\in\Z_m}V_d^2\geq \sum_{0<\abs{h}<\ell}(\ell-\abs{h})^2\mathcal E(f)
\gtrsim \ell^3\mathcal E(f)\geq \beta\ell^3m^3.
$$
This still includes the zero step. For $d=0$, we have $S_0(x)=\ell f(x)$, and
therefore $V_0=\ell^2\sum_{x\in\Z_m}f(x)^2\leq \alpha\ell^2m$, so
$V_0^2\leq \alpha^2\ell^4m^2$. The assumption
$\ell\lesssim \beta m/\alpha^2$, with the implicit constant chosen small enough,
makes this smaller than a fixed fraction of $\beta\ell^3m^3$. Removing the zero
step gives \eqref{eq:v-second-moment-nonzero}.
\end{proof}

\begin{lemma}[Dense modular progression]\label{lem:modular}
Assume $\mathcal E(f)\geq\beta m^3$, with $0<\beta\lesssim\alpha^2$, and choose
an integer $\ell$ satisfying
$$
\ell\sim \frac{\beta m}{\alpha^2},\qquad 1\leq \ell<m/2,
$$
where the implicit constant is chosen small enough. Then there exist
$d\in\Z_m^\times$ and $x\in\Z_m$ such that
\begin{equation}\label{eq:modular-increment}
S_d(x)=\sum_{i=0}^{\ell-1} f(x+id)\gtrsim \frac{\beta}{\alpha}\ell.
\end{equation}
\end{lemma}

\begin{proof}
Set $M=\max_{d\in\Z_m^\times}V_d$. By \eqref{eq:v-second-moment-nonzero} and
\eqref{eq:v-first-moment},
$$
\beta\ell^3m^3\lesssim \sum_{d\in\Z_m^\times}V_d^2
\leq M\sum_{d\in\Z_m}V_d\leq M\alpha\ell m^2.
$$
Thus there is a nonzero step $d$ with
\begin{equation}\label{eq:large-vd}
V_d\gtrsim \frac{\beta}{\alpha}\ell^2m.
\end{equation}

We now extract a large positive value of $S_d(x)$. Since $\abs{S_d(x)}\leq\ell$
for every $x\in\Z_m$, the definition of $V_d$ gives
$V_d\leq \ell\sum_{x\in\Z_m}\abs{S_d(x)}$. Using \eqref{eq:large-vd}, we get
$$
\sum_{x\in\Z_m}\abs{S_d(x)}\gtrsim \frac{\beta}{\alpha}\ell m.
$$
Also $\sum_{x\in\Z_m}S_d(x)=\ell\sum_{x\in\Z_m}f(x)=0$. Therefore the total
positive mass of $S_d$ is exactly half of $\sum_{x\in\Z_m}\abs{S_d(x)}$, and
hence some $x\in\Z_m$ satisfies \eqref{eq:modular-increment}.
\end{proof}

%% -------------------------------------------------------------------
\section{Finding a genuine integer progression}
%% -------------------------------------------------------------------

We now show how to pass from a density increment on a modular progression to
one on a genuine integer progression.

\begin{lemma}[Finding a genuine integer progression]\label{lem:rect}
Suppose $d\in\Z_m^\times$ and $x\in\Z_m$ satisfy
$\sum_{i=0}^{\ell-1}f(x+id)\geq\eta\ell$ for some $0<\eta\leq1$. If
$\ell\gtrsim\eta^{-2}$, then there is an arithmetic progression $P\subset[N]$
such that
$$
\abs{P}\gtrsim\eta\sqrt{\ell}
\qquad\text{and}\qquad
\frac{\abs{A\cap P}}{\abs{P}}\geq\alpha+c\eta
$$
for an absolute constant $c>0$.
\end{lemma}

\begin{proof}
We find a short subprogression of the original progression whose step is
represented by a small integer, then intersect it with $[N]$.

Apply the pigeonhole principle to the $\lfloor\sqrt{\ell}\rfloor+1$ elements
$0,d,2d,\ldots,\lfloor\sqrt{\ell}\rfloor d$ in $\Z_m$: two of these have a
difference represented by an integer of size $\lesssim m/\sqrt{\ell}$. Hence
there is an integer $q$, with $1\le q\le\sqrt{\ell}$, such that $qd$ is
represented by an integer $s$ with $\abs{s}\lesssim m/\sqrt{\ell}$. Since
$m\sim N$, after enlarging the absolute constant in $\ell\gtrsim\eta^{-2}$, we
may assume $\abs{s}<N$.

Taking every $q$th term of the original progression gives $q$ subprogressions,
each with step represented by $s$. Divide each of these into consecutive progressions
of length $K\sim\sqrt{\ell}$, choosing the implicit constant in $K$ small
enough that each subprogression has length at least $K$ and every progression travels
distance $<m/10$ before reduction modulo $m$. Merging any incomplete final
progression into its predecessor, we may assume every progression has length
$\sim\sqrt{\ell}$. These progressions still partition the original $\ell$ terms.
Since the original average of $f$ is at least $\eta$, one block $\widetilde P$
satisfies
\[
    \sum_{u\in\widetilde P} f(u)\geq \eta\abs{\widetilde P}.
\]

It remains to intersect $\widetilde P$ with $[N]$. On the integer line, the
copies of $[N]$ modulo $m$ are separated by gaps of length $m-N>3m/4$, while
$\widetilde P$ travels distance $<m/10$. Thus $\widetilde P$ cannot meet
$[N]$, leave $[N]$, and later meet $[N]$ again. Hence the terms of
$\widetilde P$ lying in $[N]$ form one consecutive subprogression; call it
$P$. Since $f=0$ outside $[N]$, passing from $\widetilde P$ to $P$ does not
change the sum of $f$. As $f\leq1$, we get
$\abs{P}\geq\eta\abs{\widetilde P}\gtrsim\eta\sqrt{\ell}$, and the average of
$f$ on $P$ is at least $\eta$. Finally, $P$ is an ordinary integer arithmetic
progression: all its terms lie in $[N]$, and its step is represented by the
integer $s$ with $\abs{s}<N$. Since $f=\one_A-\alpha$ on $[N]$,
\[
    \frac{\abs{A\cap P}}{\abs{P}}
    =\alpha+\frac{1}{\abs{P}}\sum_{u\in P}f(u)
    \geq \alpha+\eta.
\]
This completes the proof.
\end{proof}

Combining this with Lemma~\ref{lem:energy} gives the following density increment.

\begin{proposition}[Density increment]\label{prop:increment}
Let $A\subset[N]$ have density $\alpha$, and suppose that $A$ contains no
nontrivial three-term arithmetic progression. If $N\gtrsim\alpha^{-32}$, then
there is an arithmetic progression $P\subset[N]$ such that
$$
\abs{P}\gtrsim\alpha^{16}N^{1/2}
\qquad\text{and}\qquad
\frac{\abs{A\cap P}}{\abs{P}}\geq\alpha+c\alpha^{11}.
$$
\end{proposition}

\begin{proof}
By Lemma~\ref{lem:energy}, $\mathcal E(f)\gtrsim\alpha^{12}m^3$. Put
$\beta\sim\alpha^{12}$ and choose
$\ell\sim\beta m/\alpha^2\sim\alpha^{10}m$. The hypothesis
$N\gtrsim\alpha^{-32}$ ensures that such an integer $\ell$ satisfies
$1\leq\ell<m/2$. Lemma~\ref{lem:modular} gives $d\in\Z_m^\times$ and
$x\in\Z_m$ with $\sum_{i=0}^{\ell-1}f(x+id)\gtrsim \alpha^{11}\ell$. Set
$\eta\sim\alpha^{11}$. Again using $N\gtrsim\alpha^{-32}$, we have
$\ell\gtrsim\eta^{-2}$. Lemma~\ref{lem:rect} gives an arithmetic progression
$P\subset[N]$ with density increment $\gtrsim\alpha^{11}$ and length
$$
\abs{P}\gtrsim \eta\sqrt{\ell}\gtrsim \alpha^{11}\alpha^5m^{1/2}
\gtrsim \alpha^{16}N^{1/2}.
$$
\end{proof}

%% -------------------------------------------------------------------
\section{Proof of Roth's theorem and the quantitative bound}
%% -------------------------------------------------------------------

\begin{proof}[Proof of Theorem~\ref{thm:roth}]
Fix $\alpha_0>0$ and suppose, for contradiction, that there are arbitrarily
large $N$ and 3AP-free sets $A\subset[N]$ with $\abs{A}\geq\alpha_0N$.
Applying Proposition~\ref{prop:increment}, then rescaling the progression it
gives back to an interval, produces a new 3AP-free set with density increased
by at least $c\alpha_0^{11}$. Repeating this, the density exceeds $1$ after
$O(\alpha_0^{-11})$ steps.

It remains only to ensure that the iteration can be run for that many steps.
At each step the ambient length is replaced by a progression of length at least
$c\alpha_0^{16}N_j^{1/2}$. Since the number of steps depends only on
$\alpha_0$, choosing the original $N$ sufficiently large in terms of $\alpha_0$
keeps every intermediate interval above the threshold in
Proposition~\ref{prop:increment}. This contradiction proves the theorem.
\end{proof}

A more careful analysis  gives an explicit double-logarithmic bound.

\begin{corollary}\label{cor:quantitative}
There is an absolute constant $C>0$ such that
$$
        r_3(N)\leq C\frac{N}{(\log\log N)^{1/11}}
$$
for all sufficiently large $N$.
\end{corollary}

\begin{proof}
Let $A\subset[N]$ be 3AP-free, and write $\alpha=\abs{A}/N$. Set
$J=\lfloor c_0\alpha^{-11}\rfloor$, where $c_0>0$ is a sufficiently small
absolute constant. If the density increment can be applied $J$ times, then each
step increases the density by at least $c\alpha^{11}$, since the densities only
increase. After $J$ steps the density would exceed $1$, which is impossible.

It remains to check that the lengths stay large enough for $J$ steps. At each
step, after rescaling the subprogression back to an interval,
Proposition~\ref{prop:increment} gives
\[
        \alpha_{j+1}\geq \alpha_j+c\alpha_j^{11},
        \qquad
        N_{j+1}\geq c\alpha_j^{16}N_j^{1/2}.
\]
Since $\alpha_j\geq\alpha$, the length estimate implies
\[
        \log N_{j+1}\geq \frac12\log N_j-C\log(e/\alpha).
\]
Iterating this crude bound gives
\[
        \log N_j\geq 2^{-j}\log N-C\log(e/\alpha).
\]
Thus, if $\log N\geq C2^J\log(e/\alpha)$, then every $N_j$ with $j\leq J$
remains larger than a fixed power of $\alpha^{-1}$, and in particular remains
above the threshold required by Proposition~\ref{prop:increment}. Since
$J\sim\alpha^{-11}$, this condition follows from
\[
        \log\log N\geq C\alpha^{-11},
\]
after increasing $C$.

Therefore, if $\log\log N$ were larger than a sufficiently large constant times
$\alpha^{-11}$, the iteration could run for $J$ steps and force the density
above $1$. Hence $\log\log N\lesssim\alpha^{-11}$, so
$\alpha\lesssim(\log\log N)^{-1/11}$. Since $\alpha=\abs{A}/N$, this gives
\[
        r_3(N)\lesssim \frac{N}{(\log\log N)^{1/11}}.
\]
This proves the corollary.
\end{proof}

%% -------------------------------------------------------------------
\section{Further remarks}
%% -------------------------------------------------------------------

Although the proof above is written without Fourier analysis, several of the steps have a simple Fourier interpretation. Let $e(x)=e^{2\pi i x/m}$ and
$\widehat f(\xi)=\sum_{x\in\Z_m} f(x)e(-\xi x)$. One then has
$\widehat R(\xi)=\abs{\widehat f(\xi)}^2$. Hence Parseval gives
$$
\mathcal E(f)=\sum_{t\in\Z_m}R(t)^2=
\frac1m\sum_{\xi\in\Z_m}\abs{\widehat f(\xi)}^4.
$$
Thus, up to the normalization convention, $\mathcal E(f)$ is the fourth power
of what is called in the modern Gowers-style literature the $U^2$-norm. In the
usual modern Fourier-analytic proofs of Roth's theorem, the implication from a small
$3$-term progression count to a large $U^2$-norm is standard.

The key non-negativity statement in Lemma~\ref{lem:pos} also has a short Fourier
explanation. If $h,k\in\Z_m^\times$ and $\lambda=kh^{-1}$, then
\begin{equation}\label{eq:fourier-positivity}
\sum_{d\in\Z_m}R(hd)R(kd)
=\sum_{t\in\Z_m}R(t)R(\lambda t)
=\frac1m\sum_{\xi\in\Z_m}
\abs{\widehat f(\xi)}^2
\abs{\widehat f(-\lambda^{-1}\xi)}^2
\geq 0.
\end{equation}
As future research directions, it would be interesting to develop Fourier-free analogues of the following
arguments:
\begin{enumerate}
    \item a proof of a polylogarithmic Roth bound of the form $r(n)\lesssim n(\log n)^{-c}$ for $c>0$;
    \item the energy-increment proof of Roth's theorem;
    \item Gowers' argument \cite{Gowers98} for $4$-term arithmetic progressions.
\end{enumerate}


\begin{thebibliography}{DHJ}

\bibitem{DHJ}
D.~H.~J. Polymath,
\emph{A new proof of the density Hales--Jewett theorem},
Ann. of Math.\ (2) \textbf{175} (2012), no.~3, 1283--1327.

\bibitem{Furstenberg77}
H.~Furstenberg,
\emph{Ergodic behavior of diagonal measures and a theorem of Szemer\'edi on
arithmetic progressions},
J. Anal. Math.\ \textbf{31} (1977), 204--256.

\bibitem{Gowers98}
W.~T. Gowers,
\emph{A new proof of {Szemerédi's} theorem for arithmetic progressions of length four},
Geom. Funct. Anal.\ \textbf{8} (1998), no.~3, 529--551.

\bibitem{GRS}
R.~Graham, B.~Rothschild, and J.~Spencer,
\emph{Ramsey theory},
Wiley, New York, 1990.

\bibitem{Green}
B.~Green,
\emph{Arithmetic progressions at the Journal of the LMS},
J. Lond. Math. Soc. (2) \textbf{113} (2026), no.~3, Paper No.~e70483.

\bibitem{P}
S.~Peluse,
\emph{Recent progress on bounds for sets with no three terms in arithmetic
progression},
Ast\'erisque \textbf{438} (2022), Exp. No.~1196, 547--581.

\bibitem{RS78}
I.~Z. Ruzsa and E.~Szemer\'edi,
\emph{Triple systems with no six points carrying three triangles},
Combinatorics (Keszthely, 1976), Colloq. Math. Soc. J\'anos Bolyai \textbf{18},
North-Holland, 1978, pp.~939--945.

\bibitem{Roth53}
K.~F. Roth,
\emph{On certain sets of integers},
J. London Math. Soc.\ \textbf{28} (1953), 104--109.

\bibitem{Szemeredi75}
E.~Szemer\'edi,
\emph{On sets of integers containing no $k$ elements in arithmetic progression},
Acta Arith.\ \textbf{27} (1975), 199--245.

\end{thebibliography}
\end{document}